\documentclass[reqno, a4paper]{amsart}
\usepackage{amssymb, mathdots}
\usepackage{mathabx}
\usepackage{mathrsfs}
\usepackage[utf8]{inputenc}
\usepackage{amsmath,  graphicx, tensor}
\usepackage{tikz}
\usetikzlibrary{matrix, arrows.meta}
\usetikzlibrary{cd}

\usepackage{enumerate}
\usepackage{hyperref}
\DeclareSymbolFont{bbold}{U}{bbold}{m}{n}
\DeclareSymbolFontAlphabet{\mathbbold}{bbold}
 \def\psp#1#2%
  {\mathop{}%
   \mathopen{\vphantom{#2}}^{#1}%
   \kern-\scriptspace%
    \hskip -0.3mm{#2}} 
 \def\psb#1#2%
  {\mathop{}%
   \mathopen{\vphantom{#2}}_{#1}%
   \kern-\scriptspace%
    \hskip -0.0mm{#2}} 
 \def\pscr#1#2#3%
  {\mathop{}%
   \mathopen{\vphantom{#3}}^{#1}_{#2}%
   \kern-\scriptspace%
    \hskip -0.3mm{#3}} 

\def\C{{\mathbb C}}

\def\N{{\mathbb N}}

\def\P{{\mathbb P}}

\def\R{{\mathbb R}}

\def\cringle{\mathaccent23}
\def\union{\mathop{\bigcup}}

\def\textmap#1{\mathop{\vbox{\ialign{
                                  ##\crcr
      ${\scriptstyle\hfil\;\;#1\;\;\hfil}$\crcr
      \noalign{\kern 1pt\nointerlineskip}
      \rightarrowfill\crcr}}\;}}
\def\bigtextmap#1{\mathop{\vbox{\ialign{
                                  ##\crcr
      ${\hfil\;\;#1\;\;\hfil}$\crcr
      \noalign{\kern 1pt\nointerlineskip}
      \rightarrowfill\crcr}}\;}}
      
\newcommand{\cal}{\mathcal}
\def\textlmap#1{\mathop{\vbox{\ialign{
                                  ##\crcr
      ${\scriptstyle\hfil\;\;#1\;\;\hfil}$\crcr
      \noalign{\kern-1pt\nointerlineskip}
      \leftarrowfill\crcr}}\;}}

\def\g{{\mathfrak g}}

\theoremstyle{remark}

\newtheorem*{pb}{Problem}

\newtheorem*{cl}{Claim}

\newtheorem{sz}{Satz}[section]
\theoremstyle{remark}
\newtheorem{re}[sz]{Remark} 
\theoremstyle{plain}
\newtheorem{thry}[sz]{Theorem}
\newtheorem{pr}[sz]{Proposition}
\newtheorem{co}[sz]{Corollary}
\newtheorem{dt}[sz]{Definition}
\newtheorem{lm}[sz]{Lemma}

\def\End{\mathrm {End}}
\def\Aut{\mathrm {Aut}}

\def\GL{\mathrm {GL}}
\def\gl{\mathrm {gl}}

\def\id{ \mathrm{id}}

\def\ad{\mathrm {ad}}

\newcommand\smvee{{\hskip -0.1ex \raise 0.2ex\hbox{$\scriptscriptstyle\vee$}}\hskip -0,3ex}

\newcommand{\longhookrightarrow}{}
\DeclareRobustCommand{\longhookrightarrow}{\lhook\joinrel\relbar\joinrel\rightarrow}

\def\Ad{\mathrm{Ad}}
\def\trp#1{\tensor[^{\mathrm{t}}]{#1}{}}
\def\edf{\coloneq}

\begin{document}

\title{Holomorphic bundles on complex manifolds with boundary}

\author{Andrei Teleman }
\address{Aix Marseille Univ, CNRS, I2M, Marseille, France.}
\email{andrei.teleman@univ-amu.fr}
\begin{abstract}
Let $\Omega$ be a complex manifold, and let $X\subset \Omega$  be an open submanifold whose closure $\bar X$ is a (not necessarily compact) submanifold with smooth boundary.	

 Let $G$ be a complex Lie group, $\Pi$ be a differentiable principal $G$-bundle on $\Omega$ and $J$ a formally integrable bundle almost complex structure on the restriction $\bar P\coloneq \Pi|_{\bar X}$. We prove that, if the boundary of $\bar X$ is strictly pseudoconvex,  $J$ extends to a holomorphic structure on the restriction of $\Pi$ to a neighborhood of $\bar X$  in $\Omega$. This answers positively and generalizes a problem stated in \cite{Do}. We obtain a gauge theoretical interpretation of the quotient  ${\cal C}^\infty(\partial \bar X,G)/{\cal O}^\infty(\bar X,G)$ associated with any compact Stein manifold with boundary $\bar X$ endowed with a Hermitian metric. 
 
For a fixed differentiable $G$-bundle $\bar P$ on a complex manifold $\bar X$ with non-pseudoconvex boundary, we study the set of formally integrable almost complex structures on $\bar P$ which admit formally holomorphic local trivializations at boundary points. We give an example where a ``generic"   formally integrable almost complex on $\bar P$ admits formally holomorphic local trivializations at {\it no} boundary point, whereas the set of formally integrable almost complex structures which admit formally holomorphic local trivializations at {\it all} boundary points is dense.
 
 \end{abstract}
 
 \thanks{I am indebted to Mauro Nacinovich for a very useful exchange of mails on the collar neighborhood theorem. I thank Matei Toma for interesting discussions which led me to the problems considered in this article. }

\maketitle

\tableofcontents

\section{Introduction}

A complex manifold with boundary is a manifold with boundary   $\bar X=X\cup\partial \bar X$ endowed with a smooth (up to boundary) almost complex structure (ACS) $j$ which is   {\it formally integrable}, i.e.  whose Neihenhuis tensor vanishes. Note that the "up to the boundary" version of the  Newlander-Nirenberg theorem does not hold in general. More precisely the formal integrability of  $j$ does not guarantee the existence of $j$-compatible charts around the boundary points \cite{Hi}. For this reason Hill uses the terminology "complex manifold with abstract boundary" for a pair $(\bar X,j)$ as above \cite{Hi}.     

A smooth map $f:(\bar X,\partial\bar X)\to (\bar Y, \partial\bar Y)$ between complex manifolds with boundary, or a smooth map $f:\bar X\to Y$ from a complex manifold with boundary to a complex manifold, will be called {\it formally holomorphic} if it is pseudo-holomorphic, i.e. if its differential at any point commutes with the two almost complex structures. Such a map is holomorphic in the classical sense (hence analytic) at any $x\in X$, but, in general, not at a boundary point; this justifies the chosen terminology {\it formally} holomorphic. If $Y=\C$, the formally holomorphy condition becomes $\bar\partial f=0$.

 Let $G$ be a complex Lie group   and $p:\bar P\to \bar X$ be a differentiable principal $G$-bundle  on $\bar X$. The boundary $\partial\bar P$ of $\bar P$ is the pre-image $p^{-1}(\partial\bar X)$. By definition, a bundle  ACS on $\bar P$ is an ACS   which makes the projection $\bar P\to \bar X$ and the $G$-action $\bar P\times G\to \bar P$  pseudo-holomorphic.  In the particular case  $G=\GL(r,\C)$, the data of a bundle ACS  $J$ on $\bar P$ is equivalent to the data of a Dolbeault operator  $\delta:A^{0}(\bar X,E)\to A^{01}(\bar X,E)$ on the associated vector bundle $E$, and the formal integrability condition for $J$ is equivalent to the formal integrability condition $\delta^2=0$ for $\delta$. \vspace{1mm}
 
 The main goal of this article is the following extension theorem:

 \newtheorem*{Th1}{Theorem \ref{ExtToCollar}}
 \begin{Th1}
Let $\Omega$ be a complex manifold, and let $X\subset \Omega$  be an open submanifold whose closure $\bar X$ is a   submanifold with smooth boundary.	 Let $G$ be a complex Lie group, $\Pi$ be a differentiable principal $G$-bundle on $\Omega$ and $J$ be a formally integrable bundle ACS on the restriction $\bar P\coloneq \Pi|_{\bar X}$

If the boundary of $\bar X$ is strictly pseudoconvex,  there exists an open neighborhood $\Omega'$ of $\bar X$ in $\Omega$ and an integrable bundle ACS $J'$  on $\Pi|_{\Omega'}$ which extends $J$.  
 \end{Th1}

I came to this statement thinking  of the following problem stated in \cite[p. 102]{Do}:  

\begin{pb}\label{DoPb} (Donaldson's problem) Let $X\subset\C^n$  be a relatively compact domain with smooth, strictly pseudoconvex  boundary, and let $E$ be a topologically trivial holomorphic vector bundle  on $\bar X$. Prove that $E$ has a global smooth trivialization on $\bar X$ which is holomorphic on $X$.
\end{pb} 	

  In the special case $n=2$ the claim is proved in  \cite[Appendix A]{Do}.  
    As noted at the beginning of Donaldson's proof, the claim follows easily from Grauert's classification theorem for holomorphic bundles on Stein manifolds if one can prove that $E$ extends to an open neighborhood of $\bar X$ in $\C^n$.

Therefore a  special case of Theorem \ref{ExtToCollar} gives the needed extension property, so it solves Donaldson's problem. 

  The proof of Theorem \ref{ExtToCollar} is inspired by the beautiful proof of the collar neighborhood theorem of \cite{HiNa} for complex manifolds with strictly pseudoconvex boundary. I am deeply indebted  to Mauro Nacinovich for answering in detail my questions   and clearing up a technical difficulty   in the proof  of \cite[Theorem 1]{HiNa}; the same difficulty occurs  in general   when one uses Zorn's lemma to prove extension theorems for objects defined on manifolds with boundary. The Appendix   explains in detail this difficulty in a  general framework which includes both extensions problems; the main result of this section (Lemma \ref{Nac})  is essentially due to  Nacinovich \cite{Nac}. \vspace{1mm}      
  
  Note that the conclusion of Theorem \ref{ExtToCollar} holds for any, not necessarily reductive, complex Lie group $G$ and for any  open, not necessarily relatively compact, submanifold $X$ with smooth, strictly pseudoconvex boundary of a complex manifold $\Omega$.    Moreover,   our formal integrability condition for a bundle ACS does not require the existence of  formally holomorphic local trivializations of $\bar P$ at boundary points $x\in\partial\bar X$. This {\it strong integrability condition}  is required in Donaldson's definition of a holomorphic bundle on a complex manifold with boundary \cite[p. 91]{Do} so, in fact,  it is part of the hypothesis of the original Donaldson's problem. On the other hand, under the assumptions  of Theorem \ref{ExtToCollar}, this conditions is superfluous, because  (see section \ref{LocTr}):
  
  \newtheorem*{Th2}{Proposition \ref{Co1}}
  \begin{Th2}
  Let $\bar X$ be a complex manifold with boundary, $G$ be a complex Lie group and $p:\bar P\to \bar X$ be a principal $G$-bundle on $\bar X$ endowed with a formally integrable bundle ACS $J$. Let $x\in \partial\bar X$ and $y\in \bar P_x$. If $\bar X$ is (weakly) pseudoconvex around $x$, there exists an open neighborhood 	$U$ of $x$ in $\bar X$ and a formally holomorphic section $s:U\to \bar P$ such that $s(x)=y$.	
  \end{Th2}
  In other words {\it any} formally integrable bundle ACS on a bundle over a complex manifold with pseudoconvex boundary satisfies Donaldson's strong integrability condition.
  \vspace{1mm}

  We now recall the motivation behind Donaldson's problem: to give a gauge-theoretical interpretation of the quotient   of the group of smooth maps $\partial\bar X\to\GL(r,\C)$ by the subgroup of those maps which admit a formally holomorphic extension $\bar X\to \GL(r,\C)$ (see \cite[p. 102]{Do}). Using  Theorem \ref{ExtToCollar} and the generalization of Donaldson's  \cite[Theorem 1]{Do} to the Hermitian framework \cite{Xi}, one obtains the following generalization of this isomorphism theorem:
  \newtheorem*{Th3}{Theorem \ref{IsoGenNew}}
  \begin{Th3}
  Let $K$ be a compact Lie group and $G$ be its complexification. Let $\bar X=X\cup\partial \bar X$ be a compact Stein manifold with boundary. Endow $\bar X$ with any (not necessarily Kählerian) Hermitian metric $g$.
 
  Let  	${\cal O}^\infty(\bar X,G)$ be the group of formally holomorphic $G$-valued   maps on $\bar X$, identified with a subgroup of ${\cal C}^\infty(\partial\bar X,G)$ via the  restriction map. 
  
  There exists a natural bijection between the moduli space of boundary framed Hermitian-Yang-Mills connections on the trivial $K$-bundle on $(\bar X,g)$ and the quotient ${\cal C}^\infty(\partial\bar X,G)/{\cal O}^\infty(\bar X,G)$.	
  \end{Th3}

 Note that in this statement $\bar X$ is just an abstract complex manifold with boundary. However, by the collar neighborhood theorem mentioned above (see \cite{HiNa} for the general case and \cite[Theorem p. 706]{NO}, \cite{Oh},   \cite[Theorem 6]{Ca2} for the compact case),  the proof of Theorem \ref{IsoGenNew}   can make use of Theorem \ref{ExtToCollar}.
   %
 \vspace{1mm}
 
 The moduli spaces intervening in Theorem \ref{IsoGenNew} will play an important role in a joint project developed in collaboration with Matei Toma  dedicated to the generalization of the concept "bounded family of coherent sheaves" in non-algebraic complex geometry. We will be especially interested in the case when $\bar X$ is a compact neighborhood of a 2-codimensional analytic set in a complex $n$-manifold.  For $n\geq 3$, such a neighborhood will not be pseudoconvex in general. Therefore we come to the following two natural questions:\vspace{1mm}
 \begin{enumerate}[(Q1)]
 \item What is the role of Donaldson's strong integrability condition in the theory? \vspace{-2mm}
 	\item To what extent is the pseudoconvexity assumption in Proposition  \ref{Co1} necessary? If $\bar X$ is not assumed to be pseudoconvex, what can be said about the abundance of the ACS which satisfy this strong integrability condition within the whole space of formally integrable   bundle ACS  on a bundle $\bar P$? 
 \end{enumerate}
 \vspace{3mm}
  Concerning (Q1): A complex manifold with boundary $\bar X$ has the structure of a locally ringed space: it can be endowed with the sheaf ${\cal O}^\infty_{\bar X}$  of formally holomorphic ${\cal C}^\infty$-functions. The restriction of this sheaf to $X$ coincides with ${\cal O}_X$ and its restriction to $\partial\bar X$ always contains the constant sheaf $\underline{\C}_{\partial\bar X}$.  
  
  The  sheaf ${\cal E}^\infty$ of formally holomorphic sections of a rank $r$ vector bundle $E$ on $\bar X$ endowed with a formally integrable ACS (Dolbeault operator) is naturally a sheaf of ${\cal O}^\infty_{\bar X}$-modules. 
  \begin{re} \begin{enumerate}
  \item The following conditions are equivalent: 
  \begin{enumerate}
  	\item $E$ satisfies Donaldson's strong integrability condition.
  	\item The natural map ${\cal E}^\infty_x\to E_x$ is surjective for any $x\in \partial\bar X$.
  \end{enumerate}
 \item If one of these two equivalent conditions is satisfied,  ${\cal E}^\infty$ is locally free of rank $r$.
 \item  The assignment $E\mapsto {\cal E}^\infty$ defines an equivalence   between the groupoid of formally holomorphic vector bundles satisfying the strong integrability condition and the groupoid of locally free sheaves of ${\cal O}^\infty_{\bar X}$-modules.	
    \end{enumerate}
  \end{re}
  
  This remark shows that Donaldson's strong integrability condition for formally integrable bundle ACS is very natural: it defines the class of vector bundles which (as do holomorphic bundles in classical complex geometry) correspond to locally free sheaves on the base manifold, regarded as a ringed space.
  \vspace{1mm}  
  
  On the other hand {\it neither} the main result of \cite{Do} {\it nor} its generalization to the Hermitian framework really needs this strong integrability condition. One can see this in \cite[p. 317-318]{Xi}: a local holomorphic frame is only used to give explicit formulae for the connection matrix of a Chern connection (formula (2.2)) or for the Hermitian-Einstein flow (formula (2.6')). These explicit formulae show that the Hermitian-Einstein flow is a semi-linear strictly parabolic system. But similar (slightly more complicated) formulae are obtained using an arbitrary smooth local frame. More precisely, let $E$ be a rank $r$ vector bundle on $\bar X$ endowed with a formally integrable Dolbeault operator $\delta$,  $(f_1,\dots,f_r)$ be a smooth frame defined on an open set $U\subset\bar X$ and $\alpha\in A^{01}(U,\gl(r\C))$ the matrix valued (0,1)-form defined by
   $$
\delta(f_1,\dots,f_r)=(f_1,\dots,f_r)\alpha.
   $$ 
Putting  $H_{ij}=h(f_i,f_j)$ we see that, in the local frame $(f_1,\dots,f_r)$, the connection form of Chern connection is $H^{-1}\partial H- H^{-1}  \trp{\bar\alpha} H+\alpha$, and  the Hermitian-Einstein flow reads:
\begin{align*}
\dot H=&-2i\Lambda\bar\partial \partial H+2i\Lambda\big(\bar\partial H\wedge (H^{-1}\partial H)-\bar\partial H\wedge(H^{-1}\,\trp{\bar\alpha}\, H)+\bar\partial(\trp{\bar\alpha}) H +\,\trp{\bar\alpha}\wedge \bar\partial H\\
& -H\partial \alpha-\partial H\wedge \alpha+\trp{\bar\alpha}\, H \wedge\alpha- (H\alpha H^{-1})\wedge\partial H+(H\alpha H^{-1})\wedge \trp{\bar\alpha}\, H \big).	
\end{align*}
This is also a semi-linear, strictly parabolic system.
  \vspace{2mm} \\
Concerning (Q2): For a ${\cal C}^\infty$ vector bundle $E$ on a complex manifold with boundary $\bar X$, let $\Delta_E$ be the space of formally integrable Dolbeault operators on $E$; for $x\in\partial\bar X$ denote by  $\Delta_E^x$  the subspace of those $\delta\in\Delta_E$ which admit a formally $\delta$-holomorphic frame (satisfies the strong integrability condition) around $x$. The space we are interested in is the intersection $\Delta_E^{\rm si}\edf\bigcap_{x\in\partial\bar X} \Delta_E^x$ of formally integrable Dolbeault operators which are strongly integrable (admit formally holomorphic frames) around all boundary points.

The example below shows that, in general, on non-pseudoconvex manifolds, the strong integrability condition   defines a set which is infinite codimensional,  meagre (of first Baire category), but dense. Moreover,  in our example, a ``generic" formally integrable Dolbeault operator does {\it not} admit formally holomorphic frames  at {\it any} boundary point. 

\newtheorem*{PrIntro}{Proposition \ref{HoNew}}
\begin{PrIntro}
Let $\bar X$ be the complement of the standard ball $B\subset\C^2$	in $\P^2_\C$,   $L$ be a trivial ${\cal C}^\infty$ complex line bundle on $\bar X$. 
\begin{enumerate}
\item   The union $\bigcup_{x\in\partial\bar X}\Delta_L^x$ is  a  first  Baire category  subset of $\Delta_L$, in particular its subsets $\Delta_L^x$  (for $x\in \partial\bar X$), $\Delta_L^{\rm si}$ have the same property. 
\item  For any $x\in \partial\bar X$,  $\Delta_L^x$ is an infinite codimensional affine subspace of  $\Delta_L$.
\item $\Delta_L^{\rm si}$  is a first Baire category, infinite codimensional, but  dense  affine subspace of $\Delta_L$.
\end{enumerate}
	
\end{PrIntro}

This shows that,  in general, on non-pseudoconvex manifolds, Donaldson's strong integrability condition might not define a locally closed subset of $\Delta_E$. This complication should be taken into account in the construction of  the moduli space of boundary framed holomorphic bundles on a non-pseudoconvex $\bar X$.

The proof of  Proposition \ref{HoNew} is related to   Lewy's celebrated example of a first order differential  operator $L$ with analytic coefficients on $\R^3$  such that, for "generic" smooth second term $f$,  the "inhomogeneous equation" $Lu=f$ is non-solvable on any non-empty open set \cite{Le}, \cite{Ho}.

\section{Strong integrability at pseudoconvex boundary points}\label{LocTr}

Let   $\bar X=X\cup\partial\bar X$ be a complex $n$-manifold with  boundary. As explained in the introduction, the formal integrability of  $j$ does not guarantee the existence of formally $j$-holomorphic charts around the boundary points \cite{Hi}. By the main result of \cite{Ca}, this difficulty  vanishes if we assume    (weak) pseudoconvexity of the boundary:

Let $x\in \partial\bar X$. We'll say that $\bar X$ has  pseudoconvex boundary around $x$ if there exists  an open neighborhood $U$ of $x$ in $X$ and a real smooth non-positive function $r$ on $U$ such that $r^{-1}(0)=\partial\bar X\cap U$, $d_ur\ne0$ for any $u\in \partial\bar X\cap U$, and 	$\partial\bar\partial r(a,\bar a)\geq 0$ for any $a\in T^{10}_{u,X}\cap T_{u,\partial\bar X}^\C$, $u\in \partial\bar X\cap U$. \cite[Theorem, p. 234]{Ca} can be reformulated as follows
\begin{thry}  \label{CaTh}
Suppose that $\bar X$ has  pseudoconvex boundary around $x$.  
There exists an open neighborhood $U$ of $x$ and a formally $j$-holomorphic embedding $f:U\to \C^n$.
\end{thry}

By ``embedding" we mean  here an immersion which defines a homeomorphism on its image.  It follows that $f(U)$ is an $n$-dimensional  submanifold with boundary of $\C^n$  whose boundary (interior) is $f(\partial\bar X\cap U)$ (respectively $f(U\cap X)$) and the induced map $U\to  f(U)$ is a diffeomorphism which restricts to a biholomorphism between the interiors. 
As mentioned above,
\begin{re}
If we assume that $\bar X$ has {\it strictly} pseudoconvex boundary one can prove a stronger result: $\bar X$ admits a holomorphic collar	neighborhood, i.e. it can be embedded holomorphically in a complex manifold $\Omega$ \cite{HiNa}.
\end{re}

\begin{pr}	
\label{Co1} Let  $p:\bar P\to \bar X$ be principal $G$ bundle on $\bar X$ endowed with a formally integrable bundle ACS $J$.   
Let $x\in \partial\bar X$, $y\in \bar P_x$ and suppose that $\bar X$ has  pseudoconvex boundary around $x$.  There exists an open neighborhood 	$U$ of $x$ in $\bar X$ and a formally holomorphic section $s:U\to \bar P$ such that $s(x)=y$.
 \end{pr}
\begin{proof}
Regard $\bar P$ as complex manifold with  boundary $\partial \bar P=p^{-1}(\partial\bar X)$, and note that this boundary is pseudoconvex  around $y$.  To see this, use the pull-back of a boundary defining function $r$ for $\bar X$ around $x$. 

By Catlin's Theorem \ref{CaTh} there exists an open neighborhood $Q$ of $y$ in $\bar P$ and a formally $J$-holomorphic embedding $F:Q\to \C^N$, where $N=n+\dim(G)$. The formal holomorphy condition implies that the restriction of $F$ to any complex submanifold contained in the boundary is  holomorphic. Therefore we obtain a holomorphic embedding  $F|_{\bar P_x\cap Q}:\bar P_x\cap Q\to \C^N$. 

Let $A\subset \C^N$ be an $n$-dimensional affine subspace which  contains $F(y)$ and is transversal to  $F(\bar P_x\cap Q)$ at  this point. The pre-image via $F$ of the germ $(A,F(y))$ is the germ at $y$ of $n$-dimensional complex submanifold $S$ of $Q$ which is transversal to the fiber $\bar P_x$ at $y$. Therefore $S$ defines a holomorphic local section around $x$.
\end{proof}

\begin{co} \label{Co2}
Let $E$ be complex vector bundle on $\bar X$ endowed with a Dolbeault operator $\delta:\Gamma(E)\to \Gamma(\Lambda^{01}\otimes E)$ satisfying the formal integrability condition $\delta^2=0$. Then $E$ admits a formally $\delta$-holomorphic trivialization around any boundary 	point around which $\bar X$ has pseudoconvex boundary.
\end{co}
\begin{proof}
 The claim follows from Proposition \ref{Co1} taking into account that, for a Dolbeault operator $\delta$ on $E$, the formal integrability condition for the associated  bundle ACS on the frame bundle $\bar P_E$, reads $\delta^2=0$.
\end{proof}

\begin{re}
As the  example studied in section \ref{SIAtBoundary} shows, the existence of a formal holomorphic  trivialization	around a boundary point does not hold without  the pseudoconvexity condition (even if $\bar X$ admits a holomorphic collar neighborhood). 
\end{re}

\section{Extension of holomorphic bundles to collar neighborhoods}\label{ExtToCollarSect}

The goal of this section is the proof of the following extension theorem:
\begin{thry}\label{ExtToCollar}
Let $(\Omega,j)$ be a complex manifold, and let $X\subset \Omega$  be an open submanifold whose closure $\bar X$ is a submanifold with smooth boundary.	 Let $G$ be a complex Lie group, $\Pi\textmap{\pi}\Omega$ be a differentiable principal $G$-bundle on $\Omega$ and $J$ be  a formally integrable bundle ACS on the restriction $\bar P\coloneq \Pi|_{\bar X}$

If the boundary of  $\bar X$ is strictly pseudoconvex,  there exists an open neighborhood $\Omega'$ of $\bar X$ in $\Omega$ and an integrable bundle ACS $J'$ on $\Pi|_{\Omega'}$ which extends $J$.  
\end{thry}

Our proof is inspired by the proof of the collar neighborhood theorem for complex manifolds with strictly pseudo-convex boundary \cite{HiNa}, but uses a new ingredient: the étale space associated with a sheaf of sets and  Godement's theorem \cite[Théorème 3.3.1 p. 150]{Go}, which yields a  general extension principle for a section defined on a subset of the base of an étale space.  This allows us to give a simple construction of the inductively ordered set (to which Zorn's lemma is applied), which does not need factorization by an equivalence relation. Note also that in the proof we will point out the role of Lemma \ref{Nac}  to overcome the difficulty explained  in the appendix.
\vspace{2mm}

We start with a remark which will allow us to control the derivatives of a bundle ACS $J$ in a global way, without making use of charts and local trivializations.

Denoting by $V\subset T_\Pi$ the vertical tangent subbundle,  we obtain the short exact sequence 
$$0\to V\stackrel{j}{\longhookrightarrow} T_{  \Pi}\textmap{\pi_*} \pi^*(T_{\Omega})\to 0
$$
of   $G$-vector bundles on $\Pi$. Factorizing by $G$, we obtain the  short exact sequence
\begin{equation}\label{At}
0\to \ad(\Pi)\edf \Pi\times_{\Ad}\g\stackrel{\tilde j}{\longhookrightarrow} Q\edf T_{\Pi}/G\textmap{\tilde\pi_*}  T_{\Omega}\to 0	
\end{equation}
of  vector bundles on $\Omega$ (compare with \cite[Theorem  1, p. 187]{At}). 
\begin{re}\label{J-sJ}
The data of a bundle ACS $J$ on $\Pi$ is equivalent to the data of a  section $s_J\in\Gamma(\Omega, \End(Q))$ with $s_J^2=-\id_Q$ which makes the bundle morphisms $\tilde j$ and $\tilde\pi_*$ fiberwise $\C$-linear. Therefore one can use the formalism explained in Section \ref{appendix} to define globally the derivatives of  a bundle ACS $J$: one uses linear connections on $T^*_\Omega$  and $\End(Q)$ and the associated differential operators $D^k$  applied to   $s_J$.
\end{re}
\vspace{0mm}
\begin{proof} (of Theorem \ref{ExtToCollar})
Let $\mathscr{J}$ be the space of germs of integrable locally defined bundle ACS 	on $\Pi$. Therefore, as a set,  $\mathscr{J}=\coprod_{x\in \Omega}\mathscr{J}_x$ where 
$$\mathscr{J}_x=\underset{\substack{x\in U\subset \Omega\\ U {\rm open}}}{\varinjlim}\mathscr{J}(U)$$
and $\mathscr{J}(U)$ denotes the set of  integrable bundle ACSs 	on $\Pi|_U$. The set $\mathscr{J}$ comes with an obvious surjection $q:\mathscr{J} \to \Omega$.
For an element ${\cal J}\in \mathscr{J}(U)$ we denote by $\hat {\cal J}:U\to \mathscr{J}$ the associated   section of $q$ given by $\hat {\cal J}(u)={\cal J}_u$.
The topology of $\mathscr{J}$ is generated by the images of the sections of this form, and this topology makes $q$   local homeomorphisms. In the terminology of \cite[Section II.1.2]{Go}  $\mathscr{J}$ is the étale space associated with the sheaf of sets on $\Omega$ defined by  $U\mapsto \mathscr{J}(U)$. Note also that, for any open set $U\subset\Omega$, the map ${\cal J}\mapsto \hat {\cal J}$ identifies $\mathscr{J}(U)$ with the set $\Gamma(U,\mathscr{J})$ of continuous sections of $\mathscr{J}$ defined on $U$ \cite[Théorème 1.2.1, p 111]{Go}.

The restriction  $J_X$    of $J$ to $\bar P|_X=\Pi|_X$ defines a continuous section $\hat J_X:X\to \mathscr{J}$, but we have {\it no} obvious extension of this section to $\bar X$, because $J$, although is defined and smooth up to the boundary, does not give integrable bundle ACSs locally (with respect to $\Omega$) around boundary points.
\vspace{2mm}

Consider the set 
$${\cal R}\coloneq \big\{(V,\alpha)|\ V \hbox{ open in }\bar X,\  X\subset V\subset\bar X, \ \alpha\in\Gamma(V,\mathscr{J}),\ \alpha|_X=\hat J_X\big \}
$$
of {\it continuous} extensions of $\hat J_X$ to open (with respect to the relative topology) subsets of $\bar X$.  ${\cal R}$ is obviously non-empty and has an obvious partial order which   satisfies the hypothesis of  Zorn's lemma: every chain in ${\cal R}$ has an upper bound. Therefore, by this lemma, there exists a maximal element $(V^{\max},\alpha^{\max})$ of ${\cal R}$.

\begin{cl}
$V^{\max}=\bar X$.	
\end{cl}

The first step for proving the claim  uses \cite[Théorème 3.3.1 p. 150]{Go}. According to this theorem\footnote{The hypothesis of Godement's theorem requires the existence of fundamental system of paracompact neighborhoods of $V^{\max}$ in $\Omega$. Note that, in general, an open subspace of a paracompact space is not necessarily paracompact. In this article all manifolds are supposed second-countable by definition. This condition implies paracompactness and is ``hereditary" with respect to open embeddings. }, there exists an open set $O$ of $\Omega$ containing $V^{\max}$ and ${\cal J}\in \mathscr{J}(O)$ such that $\hat{\cal J}|_{V^{\max}}=\alpha^{\max}$. The maximality property of $(V^{\max},\alpha^{\max})$ implies %
\begin{equation}\label{IntersO}
\bar X\cap O =V^{\max}.	
\end{equation}
Note now that $J|_X= {\cal J}|_X$, so the corresponding sections (see Remark \ref{J-sJ})
$$ s_J\in\Gamma(\bar X,\End(Q)),\  s_{\cal J}\in\Gamma(O,\End(Q))$$
agree on $X$, and the  results of Section \ref{appendix} apply.  According to Lemma \ref{Nac}  we may suppose (by replacing $O$ by a smaller open set still satisfying (\ref{IntersO}), if necessary) that the triple $(s_J,O,s_{\cal J})$ satisfies property {\bf P}$_k$ for any $k\geq 0$. We will need this later in the proof.    \vspace{2mm}

Assume now by reductio ad absurdum that the claim is false and let $x_0\in \bar X\setminus V^{\max}$. Since $V^{\max}\supset X$, we have $x_0\in \partial\bar X$, so, by (\ref{IntersO}), it follows
\begin{equation}\label{x}
x_0\in \partial \bar X\setminus O.	
\end{equation}

  By  Whitney's theorem there exists a smooth non-negative real function $\varphi$ on $\Omega$ such that $\varphi^{-1}(0)=\Omega\setminus O$. 
  
  Let  $B$ be a relatively compact open neighborhood of $x_0$ in $\Omega$ and $\rho:B\to \R$ be a smooth, strictly  plurisubharmonic defining function for $X\cap B$.
   in other words $\rho^{-1}((-\infty,0))=X\cap B$, $\rho^{-1}(0)=\partial\bar X\cap B$ and $d\rho$ is nowehere vanishing on this intersection.  We may suppose that  $\rho$ has a smooth extension around $\bar B$ which is a strictly  plurisubharmonic submersion. 

For sufficiently small $\varepsilon>0$ the function  $\rho_\varepsilon\coloneq \rho-\varepsilon\varphi\in {\cal C}^\infty(B,\R)$ will still be  strictly  plurisubharmonic submersion, so $X_\varepsilon^B\coloneq \rho_\varepsilon^{-1}((-\infty,0))$ will be the interior of a strictly pseudoconvex manifold   $\bar X_\varepsilon^B \coloneq \rho_\varepsilon^{-1}((-\infty,0])$ with  boundary  $  \partial\bar X_\varepsilon^B= \rho_\varepsilon^{-1}(0)$. The superscript $B$ on the left emphasizes that the closure is taken with respect to the relative topology of $B$. \vspace{1mm}

Note that the submanifolds $X_\varepsilon^B$, $\bar X_\varepsilon^B$ of $B$ have the properties
\begin{equation}\label{WB}
\bar X\cap O\cap B\subset X^B_\varepsilon\subset O\cap B	.
\end{equation}
\begin{equation} \label{WBnew}
\bar X_\varepsilon^B	\subset (\bar X\cup O)\cap B.
\end{equation}
Indeed, for the first inclusion in (\ref{WB}) note that, for a point $w\in \bar X\cap O$, we have $\varphi(w)>0$ because, by construction, $\varphi$ is positive on $O$, and $\rho(w)\leq 0$  because $x\in  \bar X$. For the second inclusion  in (\ref{WB}) note that, if a point $x\in   X^B_\varepsilon$ does not belong to $O$, then $\varphi(x)=0$, so the condition $\rho_\varepsilon(x)< 0$ becomes $\rho(x)< 0$ which implies $x\in X\subset O$. A similar argument proves  (\ref{WBnew}). 

Formulae (\ref{WB}), (\ref{WBnew}) show that $X^B_\varepsilon\subset O\cap B$ and 
$$X\cap B\subset (\bar X\cap B)\cap (O\cap B)\subset X^B_\varepsilon\subset \bar X^B_\varepsilon\subset (\bar X\cap B)\cup (O\cap B), $$
so $\bar X^B_\varepsilon$ interpolates between $\bar X\cap B$ and $O\cap B$ in the sense of Definition \ref{interpdef} given in the Appendix (see Fig. \ref{picture}).
We know that the triple  $(O\cap B, s_{\cal J}|_{O\cap B})$ satisfies property {\bf P}$_k$ for any $k\geq 0$. By Remark \ref{smooth-extensions} explained in the appendix it follows that there  exists a {\it smooth} almost complex structure $J_\varepsilon$ on $\Pi|_{\bar X^B_\varepsilon}$ which agrees with $J$ on $\bar  X\cap B$ and with ${\cal J}$ on $X^B_\varepsilon$.  Note that, as explained in the appendix, the smoothness of  $J_\varepsilon$ does not follow using only the smoothnesses  of ${\cal J}$ and $J$.   

Note that $J_\varepsilon$ is formally integrable, because $J$ and ${\cal J}$ have this property.  
We have $x_0\in   \partial\bar X_\varepsilon^B$, because, by (\ref{x}), $\rho(x)=\varphi(x)=0$.  By Proposition \ref{Co1} applied to    $(\Pi_{\bar X_\varepsilon^B},J_\varepsilon)$  there exists an open neighborhood $U$  of $x_0$ in $\bar X_\varepsilon^B$ and a smooth section  $s_\varepsilon\in \Gamma(U,\Pi)$ which is formally holomorphic with respect to $J_\varepsilon$. Since $J_\varepsilon$ agrees with  ${\cal J}$ on $X^B_\varepsilon$, it follows that $s_\varepsilon$ is ${\cal J}$ holomorphic on $X_\varepsilon^B\cap U$.

\begin{figure}[h!]
\includegraphics[scale=0.7]{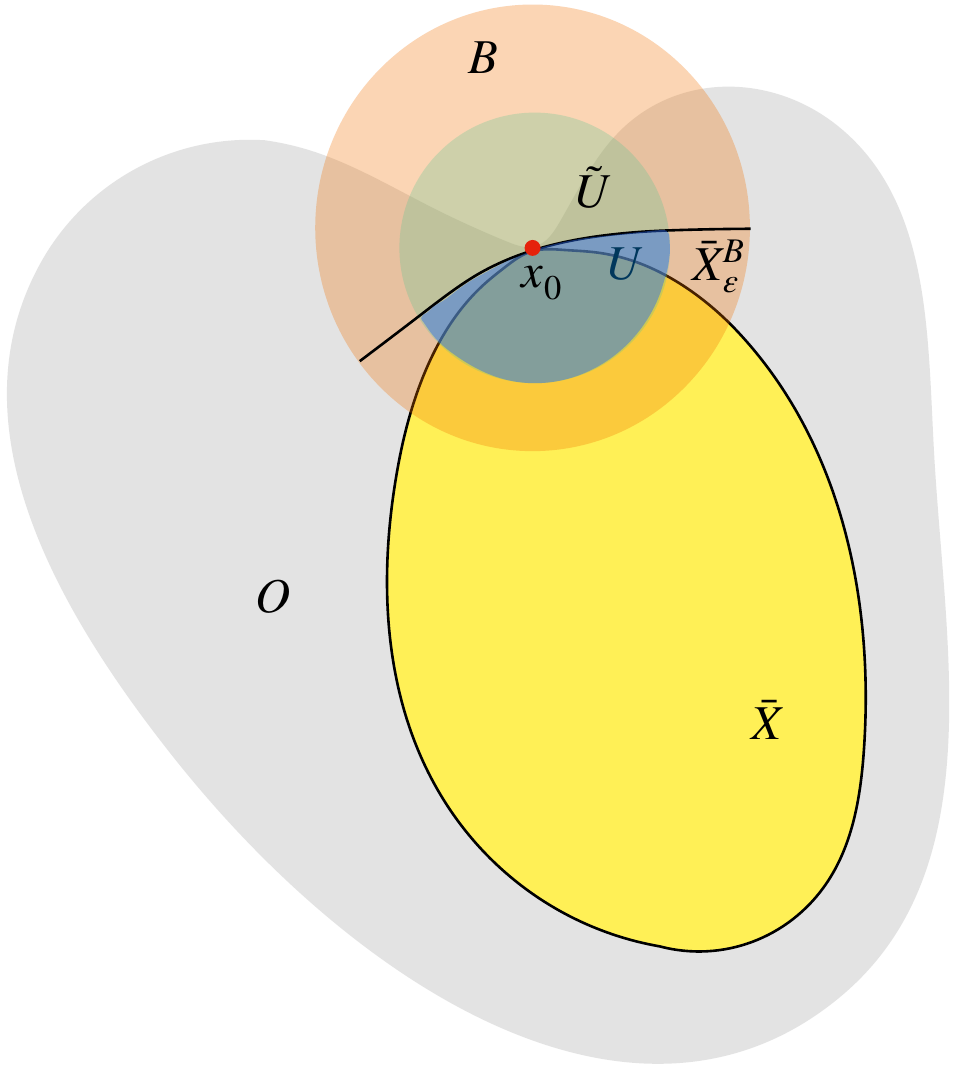}
\caption{\small The picture in the case $V^{\max}=\bar X\setminus \{x_0\}$.} 
\label{picture}	
\end{figure}

By Remark \ref{SmoothOnBarX}, the section $s_\varepsilon$ can be smoothly extended across the boundary, i.e.  there exists a  smooth extension $\tilde s_\varepsilon\in\Gamma(\tilde U,\Pi)$ of $s_\varepsilon$ to   an open set $\tilde U$ of $B$ such that $\bar X_\varepsilon^B\cap\tilde U=U$. 

Endow the restriction $\Pi|_{\tilde U}$ with the integrable bundle ACS $J_{\tilde s_\varepsilon}$  which makes $\tilde s_\varepsilon$ holomorphic.  The section   $\tilde s_\varepsilon$ is ${\cal J}$-holomorphic on $X_\varepsilon^B\cap U$, because on this set  it coincides with $s_\varepsilon$;  it follows that $J_{\tilde s_\varepsilon}$  agrees with ${\cal J}$ above this {\it  open} set, so
\begin{equation}\label{agrees}
\hat J_{\tilde s_\varepsilon}|_{X_\varepsilon^B\cap U}	=\hat {\cal J}|_{X_\varepsilon^B\cap U}.
\end{equation}

On the other hand  we have $V^{\max}\cap\tilde U\subset V^{\max}\cap B\subset X_\varepsilon^B$ (we used  (\ref{IntersO}) and (\ref{WB})), so $V^{\max}\cap\tilde U\subset X_\varepsilon^B\cap\tilde U= X_\varepsilon^B\cap\bar X_\varepsilon^B\cap\tilde U=X_\varepsilon^B\cap U$, so (\ref{agrees}) gives
$$
\hat J_{\tilde s_\varepsilon}|_{V^{\max}\cap\tilde U}=  \hat {\cal J}|_{V^{\max}\cap\tilde U}=\alpha^{\max}|_{V^{\max}\cap\tilde U}.
$$

Therefore $\hat J_{\tilde s_\varepsilon}$   defines an extension of $\alpha^{\max}$ on $V^{\max}\cup \tilde U$, so on $V^{\max}\cup (\tilde U\cap\bar X)$; this latter extension is continuous, because  $\alpha^{\max}$ and $\hat J_{\tilde s_\varepsilon}$ are continuous and $V^{\max}$, $\tilde U\cap\bar X$ are  open in $\bar X$. Since  $x_0\in V^{\max}\cup (\tilde U\cap\bar X)$, this contradicts the maximality of $\alpha^{\max}$, and the claim is proved. \vspace{2mm}

We now know that $V^{\max}=\bar X$. But then $\Omega'\coloneq O$ is an open neighborhood of $\bar X$ in $\Omega$, and $J'\coloneq{\cal J}$ is an integrable bundle ACS on $\Pi_{\Omega'}$ which extends $J$.
\end{proof}

As explained in the introduction, the motivation behind the extension Theorem \ref{ExtToCollar} is the following generalization of Donaldson's gauge theoretical interpretation of the quotient ${\cal C}^\infty(\partial\bar X,\GL(r,\C))/{\cal O}^\infty(\bar X,\GL(r,\C))$ associated to a compact, strictly pseudoconvex submanifold with boundary in $\C^n$ (see \cite[p. 102]{Do}):

 \begin{thry}\label{IsoGenNew}
 Let $K$ be a compact Lie group and $G$ be its complexification. Let $\bar X=X\cup\partial \bar X$ be a compact Stein manifold with boundary. Endow $\bar X$ with any (not necessarily Kählerian) Hermitian metric $g$.
 
  Let  	${\cal O}^\infty(\bar X,G)$ be the group of formally holomorphic $G$-valued   maps on $\bar X$, identified with a subgroup of ${\cal C}^\infty(\partial\bar X,G)$ via the  restriction map. 
  
  There exists a natural bijection between the moduli space of boundary framed Hermitian-Yang-Mills connections on the trivial $K$-bundle on $(\bar X,g)$ and the quotient ${\cal C}^\infty(\partial\bar X,G)/{\cal O}^\infty(\bar X,G)$.

 \end{thry}
 
  We recall from \cite{Do} that a (differentiably trivial) boundary framed holomorphic bundle on $\bar X$ is a pair $(J,\theta)$ consisting of a formally integrable bundle ACS on the trivial ${\cal C}^\infty$-bundle $\bar X\times G$ and a smooth section (or, equivalently, a trivialization) of the restriction of this bundle to $\partial\bar X$. Let ${\cal M}$ be the moduli space of pairs $(J,\theta)$ as above, modulo the natural action of the gauge group $\Aut(\bar X\times G)={\cal C}^\infty(\bar X,G)$ on the space of such pairs. 
  
  Taking into account Donaldson's isomorphism theorem \cite[Theorem 1']{Do} which gives a gauge theoretical interpretation of the moduli space ${\cal M}$ in terms of boundary framed Hermitian-Yang-Mills connections (and the generalization of this theorem to the Hermitian framework), the claim of  Theorem \ref{IsoGenNew} will follow as in Donaldson's article from the following
  \begin{lm}
  The natural map $R:{\cal C}^\infty(\partial\bar X,G)/{\cal O}^\infty(\bar X,G)\to {\cal M}$ given by
  $$
  R([f])\edf[(J_0,f)],
  $$	
  where $J_0$ is the trivial holomorphic structure on $\bar X\times G$, is bijective.
  \end{lm}

 Note that $[(J_0,f)]=[(J_0,f')]$ in ${\cal M}$ if and only if $[f]=[f']$ mod ${\cal O}^\infty(\bar X,G)$, so $R$ is well defined and injective. The image of $R$ is the space of gauge classes $[(J,\theta)]$ such that $J$ admits a global formally holomorphic trivialization on $\bar X$.\vspace{1mm}
 
 Therefore, Theorem \ref{IsoGenNew} will follow from:
 \begin{lm}
 Let $J$ be a formally integrable bundle ACS on $\bar X\times G$	. Then $(\bar X\times G, J)$ admits a global formally holomorphic trivialization on $\bar X$.
 \end{lm}

 \begin{proof}
 By the collar neighborhood 	theorem \cite{HiNa}, $\bar X$ can be embedded in a complex manifold $\Omega$. $\bar X$ is a compact Stein manifold with boundary, so, by definition,  there exists a strictly pseudoconvex smooth nonpositive real function $\rho$ on $\bar X$ such that $ \partial\bar X=\rho^{-1}(0)$ and  $0$ is a regular value of $\rho$. Let $\tilde\rho$ be any smooth extension of $\rho$ to a sufficiently small open neighborhood $\Omega'$ of $\bar X$ in $\Omega$ such that $\tilde \rho$ is positive on $\Omega'\setminus \bar X$. Therefore the fibre $\tilde\rho^{-1}(0)=\tilde\rho^{-1}(0)= \partial\bar X$ is compact and $0$ is a regular value of $\tilde\rho$. By the bicollar neighborhood theorem for compact real hypersurfaces (see for instance \cite[Theorem 2.32, p. 68]{Ma}) there exists $\eta>0$, an open neighborhood $U$ of $\partial\bar X$ in $\Omega'$ and a diffeomorphism $h: \partial\bar X\times(-\eta,\eta)\to U$ such that $h(x,0)=x$   and $\tilde\rho(h(x,t))=t$ for any $(x,t)\in  \partial\bar X\times (-\eta,\eta)$. Choose $\varepsilon\in(0,\eta)$ sufficiently small such that $\tilde\rho$  remains strictly pseudoconvex on the pre-image 
 $$\bar X_\varepsilon\edf(\tilde\rho|_U)^{-1}((-\infty,\varepsilon]).$$
It follows that $\bar X_\varepsilon$ is still a compact Stein manifold with boundary. 
 
 Let now $\Pi=\Omega\times G$ be the trivial ${\cal C}^\infty$-bundle on $\Omega$ and $J$ a formally integrable bundle ACS on the restriction $\Pi|_{\bar X}=\bar X\times G$. 
 
 By Theorem \ref{ExtToCollar}, there exists an open neighborhood $\Omega'$ of $\bar X$ and an integrable bundle ACS $J'$ on $\Pi|_{\Omega'}$ extending $J$. Choosing $\varepsilon$ sufficiently small, we may suppose $\Omega'=X_\varepsilon$. The pair $(\Pi|_{\Omega'},J')$ is a topologically trivial holomorphic $G$-bundle on the Stein manifold $X_\varepsilon$, so it is holomorphically trivial by Grauert's classification theorem of holomorphic bundles on Stein manifolds (see  \cite{Gr}, \cite[Theorem 8.2.1, p. 356]{For}). The restriction to $\bar X$ of a global holomorphic trivialization  $(\Pi|_{\Omega'},J')$ will be a global formally holomorphic  trivialization of $(\bar X\times G,J)$.
 
    \end{proof}

\section{Strong integrability at non-pseudoconvex boundary points}\label{SIAtBoundary}

Let $L=\bar X\times\C$ be the trivial line bundle on $\bar X$ and let $\sigma_0$ the section $x\mapsto(x,1)$. For a Dolbeault operator $\delta$ on $L$ let $\alpha_\delta$ be (0,1)-form defined by $\delta\sigma_0=\alpha_\delta\sigma_0$. The map $\delta\mapsto \alpha_\delta$ identifies the  space $\Delta_L$ of  formally integrable Dolbeault operators on $L$ with the closed subspace space $Z^{01}(\bar X)\subset A^{01}(\bar X)$ of   $\bar\partial$-closed (0,1)-forms on $\bar X$.

A local frame $\sigma=\varphi\sigma_0\in\Gamma(U,L)$ is formally holomorphic with respect to $\delta$ if and only if $\bar\partial\varphi+\alpha_\delta\varphi=0$. If $U$ is simply connected we may write $\varphi=e^\psi$, and the above formula becomes $\bar\partial\psi=-\alpha_\delta$. This shows that
\begin{re} \label{ExLm}
Let $\delta\in\Delta_L$ and $U\subset \bar X$ be a simply connected open set. 	$L$ admits a formally $\delta$-holomorphic frame on $U$ if and only if  $\alpha_\delta|_U$ is $\bar\partial$-exact.
\end{re}

Unfortunately, in general, the $\bar\partial$-Poincaré lemma does not hold at boundary points \cite{AH}.
\vspace{2mm}

 For the rest of this section  $\bar X$ will be the complement of the standard ball $B\subset\C^2$	in $\P^2_\C$. The boundary complex \cite[Chapter V]{FK}, \cite{Fo}   on the sphere $S\coloneq\partial\bar X$   reduces to
$$
0\to B^{00}={\cal C}^\infty(S^3,\C)\textmap{\bar\partial_S} B^{01}=\Gamma(S,\Lambda^{01}_H) \to 0,
$$	
where $H\edf T_S\cap jT_S$ and $\Lambda^{01}_H\subset H^*\otimes\C$ is the bundle of $(0,1)$-forms on $H$. Regarding $S^3$ as an $S^1$-bundle on $\P^1_\C$, $H$ is just the horizontal distribution of the connection which corresponds to the Chern connection of the tautological line bundle on $\P^1_\C$.
 
\begin{re} \label{SurjRem} Let $r:A^{01}(\bar X)\to B^{01}$ be the morphism induced by ``restriction to the  boundary". 
\begin{enumerate}
	\item The morphism $H^{01}(\bar X)\to H^{01}(S)$ induced by $r$ is an isomorphism.
	\item The morphism $Z^{01}(\bar X)\to  B^{01}$ induced by $r$ is surjective.
\end{enumerate} 
\end{re}
\begin{proof}
(1) The injectivity of $H^{01}(\bar X)\to H^{01}(S)$ follows by \cite[Theorem 5, p. 355]{AH} taking into account that $H^{01}(\P^2_\C)=0$.	 The surjectity follows from \cite[Theorem 6(c), p. 355]{AH} taking into account that  $H^{02}(\P^2_\C)=0$ and $H^{01}(\bar B)=0$. The latter is obtained using the vanishing of the sheaf cohomology space $H^1(B,{\cal O}_B)$ (which is obvious because $B$ is Stein) and the comparison theorem \cite[Theorem 4.3.1, p. 77]{FK} taking into account  the discussion on p. 57 in the same book. 
\vspace{2mm}\\
(2) Let $\beta\in B^{01}$. The surjectvity of $H^{01}(\bar X)\to H^{01}(S)$ shows that  there exists $\alpha\in Z^{01}(\bar X)$ and $\varphi\in B^{00}$ such that $\beta=r(\alpha)+\bar\partial_S\varphi$. Choosing a smooth extension $\psi\in {\cal C}^\infty(\bar X,\C)$ of $\varphi$ we see that
$$
\beta=r(\alpha)+\bar\partial_S (\psi|_S)=r(\alpha+\bar\partial\psi),
$$
which proves the claim.
\end{proof}
The generalized Cayley transform  
$$
C:\C^2\setminus(\{-1\}\times\C)\to\C^2\setminus(\{-i\}\times\C),\ C(z_1,z_2)\edf \bigg(i \frac{1-z_1}{1+z_1}\,,\,  \frac{z_2}{1+z_1}\bigg) 
$$	
identifies biholomorphically the unit ball $B\subset\C^2$ with the Siegel upper half-space
$$
{\cal U}\edf\{(w_1,w_2)\in\C^2|\ \Im(w_1)>|w_2|^2\},
$$
and  the punctured sphere $S\setminus\{(-1,0)\}$ with the real hypersurface
$$
\Sigma\edf \{(w_1,w_2)\in\C^2|\ \Im(w_1)=|w_2|^2\}\subset \C^2
$$
(the ``unbounded realization" of the 3-sphere, see \cite[p. 112]{Kr}).

Since $C$ is biholomorphic it follows that the boundary operator $\bar\partial_S$ on $S\setminus\{(-1,0)\}$ corresponds via $C$ to the boundary operator $\bar\partial_\Sigma$. Using the  diffeomorphism 
$$
k:\R\times\C\to \Sigma,\ k(t,w)=(t+i|w|^2,w)
$$
we see   \cite[p. 359-361]{AH} that $\bar\partial_\Sigma$ is given explicitly by
\begin{equation}\label{dSigma}
\bar\partial_\Sigma (\psi)=(L(\psi\circ k)\circ k^{-1}) d\bar w_2,
\end{equation}
where $L:{\cal C}^\infty(\R\times\C,\C)\to {\cal C}^\infty(\R\times\C,\C)$ is the Lewy operator \cite{Le}
$$
L\eta=\frac{\partial \eta}{\partial \bar w}-i w\frac{\partial \eta}{\partial t}.
$$
The composition 
$$h\edf k^{-1}\circ C:S\setminus \{(-1,0)\}\to\R\times\C$$
 is a chart of the sphere, and, in this chart, the operator $\bar\partial_S$ is given by
\begin{equation}\label{dS}
\bar\partial_S(\varphi)=(L(\varphi\circ h^{-1})\circ h)C^*(d\bar w_2).
\end{equation}
\begin{lm} With the notations above we have:\label{MV}
\begin{enumerate} 
\item   Let $V\subset S$ be a non-empty open subset. 	The set
$$
M^V\edf\{\beta\in B^{01}|\ \hbox{ the equation }\bar\partial_Su=\beta|_V \hbox{ has a distribution solution }u\in {\cal D}'(V)\}
$$
is a first Baire category subset of the Fréchet space $B^{01}$.
\item The union 
$$M\edf\bigcup_{\substack{\emptyset\ne V\subset S\\ V open}} M^V$$
is a first Baire category subset of the Fréchet space $B^{01}$.
\end{enumerate}

\end{lm}

\begin{proof} (1) We adapt the elegant proof of \cite[Theorem 3.2, p. 135]{Ho} to our situation explaining the fundamental ideas and the necessary changes. Note that our   $M$ is a subset of the Fréchet space $B^{01}$ of sections of the bundle $\Lambda^{01}_S$  on the compact manifold $S$,  whereas the set $M$ considered in   Hörmander's proof is a subset of the Fréchet space 
$$\dot B(\Omega)\edf \{\varphi\in {\cal C}^\infty(\Omega,\C)|\ \forall \gamma\in\N^n\,\forall \varepsilon>0\, \exists K\subset\Omega \hbox{ compact s.t.} \sup_{\Omega\setminus K}|D^\gamma\varphi|<\varepsilon\}$$
 associated with an open set $\Omega\subset\R^n$. \\

  For an open set $\omega\subset S\setminus\{(-1,0)\}$ and $N\in\N^*$ put
 %
 %
\begin{equation*}
\begin{split}
{\cal D}_N(\omega)&\edf \bigg\{u\in {\cal D}(\omega)|\ \sum_{|\gamma|\leq N}\sup \big|D^\gamma(\varphi\circ h^{-1})\big|\leq\frac{1}{N}\bigg\},\\
 {\cal D}'_N(\omega)&\edf\big\{u\in {\cal D}'(\omega)|\ \forall \varphi\in{\cal D}_N(\omega),\, |u(\varphi)|\leq 1\big\}=\\
 &=\big\{u\in {\cal D}'(\omega)|\ \forall\varphi\in{\cal D}(\omega),\,|u(\varphi)|\leq N\sum_{|\gamma|\leq N}\sup \big|D^\gamma(\varphi\circ h^{-1})\big|\big\}.		
\end{split}
\end{equation*}

A theorem of Schwartz (see for instance \cite[Theorem 10.12 p. 87]{Vo}) shows that ${\cal D}_N(\omega)$ is a neighborhood of 0 in ${\cal D}(\omega)$, so, by Alaoglu–Bourbaki's Theorem (see \cite[Theorem 4.7 p. 32]{Vo}), ${\cal D}'_N(\omega)$ is weakly compact in ${\cal D}'(\omega)$.\\

It follows that  
$$
M_N^\omega\edf \{\beta\in B^{01}|\ \exists u\in {\cal D}'_N(\omega),\, \bar\partial_S u=\beta|_\omega\}
$$
is a closed subset of $B^{01}$.  We claim that $\cringle{M}_N^\omega=\emptyset$. Indeed, note first that, taking into account  (\ref{dS}), the properties of Lewy's operator explained in \cite[p. 136]{Ho} and \cite[Theorem 3.1]{Ho}, there exists $g\in B^{01}$  with  $\mathrm{supp}(g)\subset \omega$ such that the equation $\bar\partial_S u=g$ has no solution in ${\cal D}'(\omega)$. If an interior point $\beta\in \cringle{M}_N^\omega$ existed, we could find $\varepsilon>0$ such that  $\beta+\varepsilon g\in M_N^\omega$. Since $M_N^\omega$ is symmetric and convex, we obtain $\frac{\varepsilon}{2}g\in M_N^\omega$, which contradicts the way in which $g$ has been chosen.

Now fix a non-empty open subset $\omega\subset S$ with $\bar \omega\subset V\setminus\{(-1,0)\}$. Therefore $\bar \omega$ is a compact subset of $V$ which is contained in the domain of the chart $h$. Taking into account the definition of  ${\cal D}'(V)$, it follows that for any $u\in {\cal D}'(V)$ there exists $N\in\N$ such that $u|_{\omega}\in {\cal D}'_N(\omega)$. This implies
$$
M^V\subset \union_{N\in\N} M_N^\omega,
$$
so $M^V$, being a subset of a first category subset, is itself of first category. \\ \\
(2) Use a countable basis for the topology of $S$.
\end{proof}

For an open set $U\subset\bar X$ put $Z^U\edf\{\alpha\in Z^{01}(\bar X)|\ \alpha|_U\hbox{ is $\bar\partial$-exact}\}$.
\begin{pr} \label{ZU}
The union
$$
Z\edf \bigcup_{\substack{U\subset\bar X open\\ U\cap\partial\bar X\ne\emptyset}} Z^U
$$
is a first Baire category subset of $Z^{01}(\bar X)$.
\end{pr}
\begin{proof}
Let $r:Z^{01}(\bar X)\to B^{01}$ be the 	restriction morphism. The restriction to $U\cap S$ of a $\bar\partial$-exact form on $U$ is $\bar\partial_S$-exact. This shows that, with  the notation introduced in Lemma \ref{MV}, we have
$$
Z^U\subset r^{-1}(M^{U\cap S}),\ Z\subset r^{-1}(M).
$$
The claim follows now from Lemma \ref{MV} taking into account that
\begin{enumerate}[(a)]
\item 	by Remark \ref{SurjRem} and the Open Mapping Theorem, $r$ is an open map. 
\item Since $r$ is open and continuous, the correspondence $r^{-1}$ (pre-image via $r$) commutes with interior and closure.  \end{enumerate}
\end{proof}

Let now $\Delta_L$ be the space of formally integrable Dolbeault operators on $L$. 
For $x\in\partial\bar X$ let  $\Delta_L^x$ be the subspace of those $\delta\in\Delta_L$ which admit a formally $\delta$-holomorphic frame  around $x$. The intersection $\Delta_L^{\rm si}\edf\bigcap_{x\in\partial\bar X} \Delta_L^x$ is precisely the space of formally integrable Dolbeault operators on $L$ which are strongly integrable (admit formally holomorphic frames) around all boundary points. We can now prove:

\begin{pr}\label{HoNew}
Let $\bar X$ be the complement of the standard ball $B\subset\C^2$	in $\P^2_\C$,   $L$ be a trivial ${\cal C}^\infty$ complex line bundle on $\bar X$. 
\begin{enumerate}
\item   The union $\bigcup_{x\in\partial\bar X}\Delta_L^x$ is  a  first  Baire category  subset of $\Delta_L$, in particular its subsets $\Delta_L^x$  (for $x\in \partial\bar X$), $\Delta_L^{\rm si}$ have the same property. 
\item  For any $x\in \partial\bar X$,  $\Delta_L^x$ is an infinite codimensional affine subspace of  $\Delta_L$.
\item $\Delta_E^{\rm si}$  is a first Baire category, infinite codimensional but  dense  affine subspace of $\Delta_L$.
\end{enumerate}
\end{pr}
\begin{proof}
(1) The first claim follows from  Remark   \ref{ExLm} and Proposition \ref{ZU} taking into account that $\bigcup_{x\in\partial\bar X}\Delta_E^x$ corresponds to $Z$ via the isomorphism 
$$\Delta_L\ni\delta\mapsto\alpha_\delta\in Z^{01}(\bar X).$$
(2) Taking into account Remarks \ref{ExLm}, \ref{SurjRem}, it suffices to prove that the linear subspace
$$B^x\edf\{\beta\in B^{01}|\ \hbox{ the germ of $\beta$ at $x$ is $\bar\partial_S$-exact}\}
$$
is infinite codimensional in $B^{01}$. We may suppose $x=(1,0)$ whose image via the chart $h=k^{-1}\circ C$ is the origin of $\R\times\C$. Let $\Lambda\subset B^{01}$ be the linear subspace defined by  
$$\Lambda\edf \{\beta\in B^{01}|\ \hbox{ the germ of $(\beta/C^*(d\bar w_2))\circ h^{-1}$ at 0 depends only on $t$}\}.$$

The main theorem of \cite{Le} shows that, for a function $f\in {\cal C}^\infty(\R\times\C)$ whose germ at 0 depends only on $t$, if the equation
$$Lu=f
$$
is smoothly solvable around 0, then this germ is real analytic. Taking into account formula (\ref{dS}) and Lewy's result it follows that
$$
B^x\cap\Lambda \subset \{\beta\in \Lambda|\ \hbox{ the germ of $(\beta/C^*(d\bar w_2))\circ h^{-1}$ at 0 is analytic}\},
$$
whose codimension in $\Lambda$ is infinite. Therefore $B^x$ is infinite codimensional in $B^{01}$. \vspace{3mm}\\
(3) The first idea would be to use the subspace $Z^{01}_{\rm an}(\bar X)$ of real analytic $\bar\partial$-closed forms on $\bar X$. Unfortunately, it is not clear if this space is dense in  $Z^{01}(\bar X)$. 

Instead we will will use the pre-image of the space $B^{01}_{\rm an}$ of real analytic sections of the bundle $\Lambda^{01}_H$, which  is dense in $B^{01}$. Since $r:Z^{01}(\bar X)\to B^{01}$ is continuous and open, it follows that $r^{-1}(B^{01}_{\rm an})$ is a dense subset  of $Z^{01}(\bar X)$.  

Taking into account again Remark   \ref{ExLm}  it suffices to prove that for any $\alpha\in r^{-1}(B^{01}_{\rm an})$ and any $x\in\partial\bar X$ there exists $U\subset\bar X$ open such that $x\in U$ and $\alpha|_U$ is $\bar\partial$-exact. 

Since $\alpha|_S$ is real analytic, by the first order Cauchy–Kovalevskaya theorem, there exists an open neighborhood $V$ of $x$ in $S$ and an analytic solution of the equation $\bar\partial_S u=\alpha|_S$ on $V$. Let $W$ be Stein neighborhood of $x$ in $\C^2$ such that $W\cap S\subset V$. Put
$$
W^\pm\edf\{x\in W|\ \pm(1-\|x\|^2)\geq 0\}. 
$$
Therefore 
$$
 W^+=W\cap \bar B,\ W^-=W\cap\bar X.
$$
Applying \cite[Theorem 5, p. 355]{AH} to the triple $(W,W^-,W^+)$ we see that the restriction morphism
$$
H^{01}(W^-)\to H^{01}(W\cap S)
$$
is injective.  Since the restriction of $\alpha|_S$  to $W\cap S$ is $\bar\partial_S$-exact, it follows that the restriction of $\alpha$ to $W^-$ is $\bar\partial$-exact, which proves the claim.

\end{proof}

\section{Appendix. Extending a  section defined on a manifold with boundary}
\label{appendix}

We begin with  a natural formalism for characterizing the  smoothness  sections in   vector bundles in an invariant way, without making use of charts and local trivializations. Instead we will use connections. 

Let $\Omega$ be a an $n$-dimensional differentiable manifold,  $E$ be differentiable real vector bundle on $\Omega$ and   $\nabla$,   $\nabla_E$ linear connections on $T^*_\Omega$ and $E$ respectively. 

We define  inductively differential operators  (acting on local sections of $E$)    by $D^1=\nabla_E$, $D^{k+1}=(\nabla^{\otimes k}\otimes \nabla_E)\circ D^k$. A continuous section $s:U\to E$ (on an open set $U\subset\Omega$) is of class ${\cal C}^1$ on  if and only $D^1s$ exists and is continuous. Proceeding inductively we see that $s$ is of class ${\cal C}^{k+1}$ if and only if it is of class ${\cal C}^k$ and $D^k s$ if of class ${\cal C}^1$.  This formalism is useful for  characterizing smoothness of sections on a submanifold with boundary $\bar X\subset\Omega$ obtained as the closure of an open submanifold $X\subset \Omega$.   Using \cite{Se} and a partition of unity   subordinate  to an open cover   with domains of charts we obtain:
\begin{re}\label{SmoothOnBarX}
A smooth section $s\in\Gamma(X,E)$ extends to a smooth section on $\bar X$ if and only if   $D^ks\in \Gamma(X,T^{*\otimes k}_\Omega\otimes E)$ extends continuously to 	$\bar X$ for any $k\geq 0$. This condition is equivalent to the existence of a smooth extension of $s$ on a collar neighborhood of $\bar X$ in $\Omega$.
\end{re}

 Let $X\subset\Omega$ be an open submanifold whose closure $\bar X$ is a smooth submanifold with boundary $\partial\bar X=\bar X\setminus X$ and let $s\in\Gamma(\bar X,E)$ be a smooth section defined on $\bar X$. Our problem is to extend $s$ across an open piece $\partial\bar X\cap O$ of the boundary using a {\it given} smooth section $\sigma$ defined on an open set $O\supset X$ which agrees with $s$ on $X$ (see Fig. \ref{fig2}). 
 \begin{figure}[h!]
\includegraphics[scale=0.5]{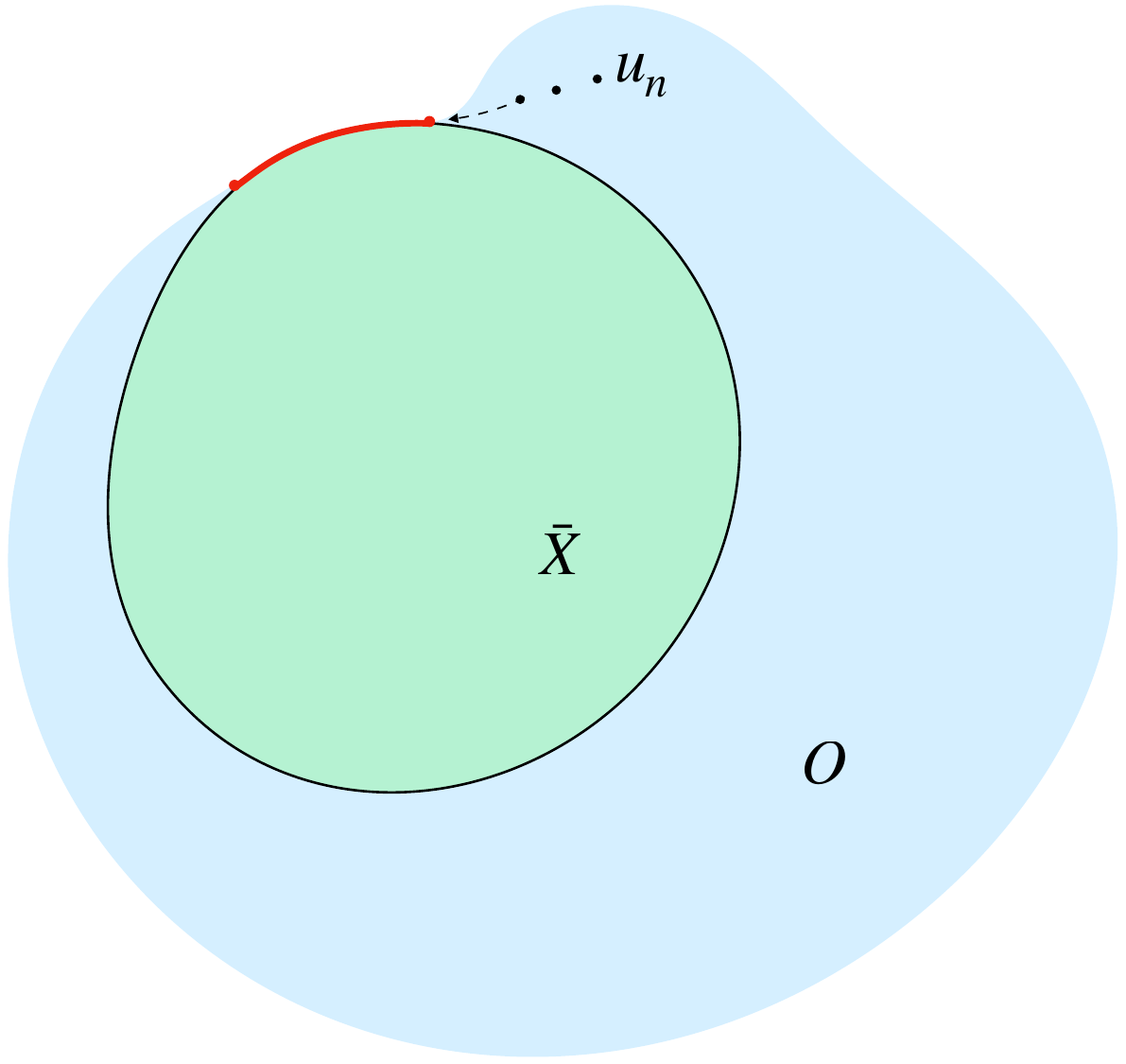}
\caption{\small A sequence $(u_n)_n$ of $O$   converging to a point $x\in\bar X\setminus O$. }
\label{fig2}	
\end{figure}

Therefore, let $O\subset \Omega$ be a open set containing $X$ and $\sigma\in\Gamma(O,E)$ be a smooth section  such that 
\begin{equation}\label{hyp}
\sigma|_X=s|_X.
\end{equation}
The closure of $X$ in $O$ is $\bar X\cap O$, so the assumption (\ref{hyp}) implies 
\begin{equation}\label{hypnew}
\sigma|_{\bar X\cap O}=s|_{\bar X\cap O}.
\end{equation}

Formula (\ref{hypnew}) suggests that we can glue together $s$ and $\sigma$ to obtain a section on the union $\bar X\cup O$ which agrees with $s$ on $\bar X$ and with $\sigma$ on $O$. The problem is that 
\newtheorem*{diff}{Difficulty} 
\begin{diff} The section $s\vee\sigma:\bar X\cup O\to E$ obtained in this way {\it might not even be  continuous}.  Continuity on $O$ (which is open, and on which $s\vee\sigma$ coincides with $\sigma$) is obvious, but not at the points of $\bar X\setminus O$. 	
\end{diff}

\begin{re}\label{RemContinSect} The section $s\vee\sigma$ is continuous if and only if the triple $(s,O,\sigma)$ satisfies the property:\end{re}
\newtheorem*{prop}{\bf P$_0$}
\begin{prop}
For every sequence $(u_n)_n$ of $O$ which converges to a point $x\in\bar X\setminus O$, the limit $\lim_{n\to\infty} \sigma(u_n)$ exists and coincides with $s(x)$ (see Fig. \ref{fig2}).
\end{prop}

 Indeed,   if {\bf P}$_0$ holds, then for any sequence  $(u_n)_n$ of $\bar X\cup O$ converging to  a point $x\in\bar X\setminus O$, the sequence $((s\vee\sigma)(u_n))_n$ admits a subsequence converging to $(s\vee\sigma)(x)$, which implies continuity at $x$. \vspace{1mm}

Our goal is to construct a {\it smooth}, not only continuous, extension of $s$ across $\bar\partial X\cap O$ using $\sigma$. Note that, by Remark \ref{SmoothOnBarX}, the assumption (\ref{hyp}) also implies 
\begin{equation}\label{hypnew-k}
D^k\sigma|_{\bar X\cap O}=	D^ks|_{\bar X\cap O} \ \forall k\geq 0,
\end{equation}
which suggests that a smoothness criterion for $s\vee\sigma$ (similar to Remark \ref{RemContinSect}) should hold.
The problem is that, in general, the union $\bar X\cup O$  has no  natural manifold structure, so smoothness is not defined for sections on this set. We are interested in extensions of $s$ to a  larger submanifold with boundary $\bar Y\supset \bar X$ which interpolates between $\bar X$ and $O$ in the following sense:
\begin{dt}\label{interpdef}
Let $Y\subset O$ be an open submanifold whose closure $\bar Y$ is a smooth submaniold with boundary. We say that $\bar Y$ interpolates between $\bar X$ and $O$ if 
\begin{equation}\label{interpeq}
X\subset \bar X\cap O\subset Y\subset \bar Y\subset \bar X\cup O.
\end{equation}

\end{dt}

\begin{figure}[h!]
\includegraphics[scale=0.5]{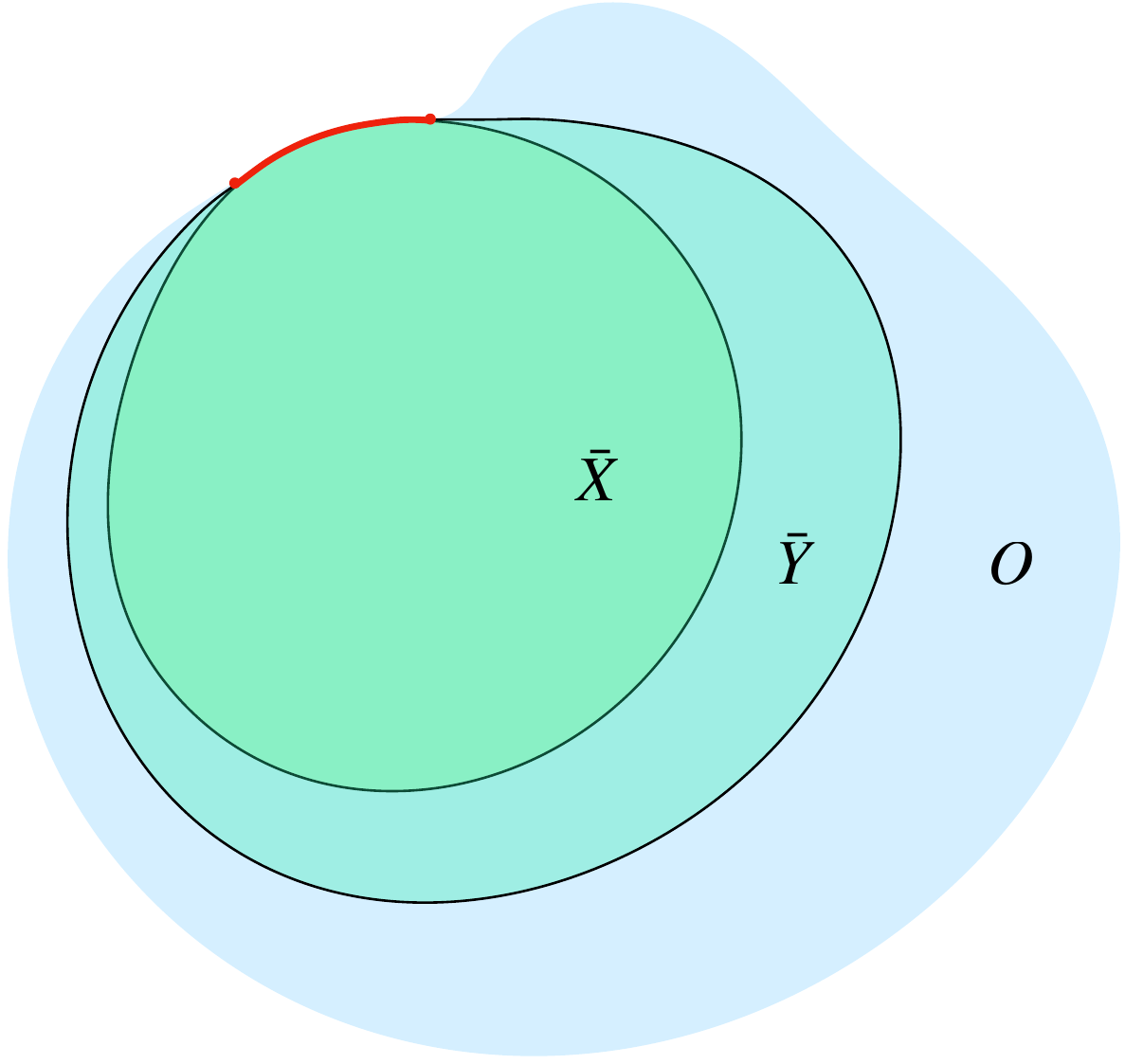}
\caption{\small A submanifold with boundary interpolating between $\bar X$ and $O$.}
\label{fig3}	
\end{figure}
Note that if $Y$ has this property, then

$$ \bar X\cap O= \bar X\cap ( \bar X\cap O)\subset \bar X\cap Y\subset \bar X\cap O,
$$
so $\bar X\cap Y=\bar X\cap O$, i.e. $Y$ and $O$ define the same relatively open subset of $\bar X$, so the same open piece $\partial\bar X\cap Y=\partial\bar X\cap O$ of the boundary.
\vspace{2mm}

Taking into account Remark \ref{SmoothOnBarX}, we obtain:

\begin{re}\label{smooth-extensions}
The restriction of   $s\vee\sigma$ 	to any submanifold with boundary $\bar Y$ interpolating between $\bar X$ and $O$ is smooth provided  for any $k\in\N$ the triple $(s,O,\sigma)$ satisfies the property 
\end{re}
\newtheorem*{propk}{\bf P$_k$}
\begin{propk}
For every sequence $(u_n)_n$ of $O$ which converges to a point $x\in\bar X\setminus O$, the limit $\lim_{n\to\infty} (D^k\sigma)(u_n)$ exists and coincides with $(D^ks)(x)$.
\end{propk}

The goal of this section is the following lemma, which is essentially due to Mauro Nacinovich \cite{Nac}.  Nacinovich's result  concerns the  extension problem for  almost complex structures considered in \cite{HiNa}; this problem can be handled  taking $E=\End(T_\Omega)$ in our formalism. The method of proof in the general case is the same.

The lemma states that, replacing $O$ by a smaller open set $O'$ with $\bar X\cap O'=\bar X\cap O$, the triple $(s,O',\sigma|_{O'})$ will satisfy {\bf P}$_k$ for any $k$, so Remark \ref{smooth-extensions} applies, giving a smooth extension of $s$ to $\bar Y$ for any submanifold with boundary $\bar Y$ interpolating between $\bar X$ and $O'$. Note that, for any such $\bar Y$ we'll have $\bar X\cap Y=\bar X\cap O'=\bar X\cap O$, so we obtain extensions across the original open piece $\partial \bar X\cap O$ of the boundary. 

\begin{lm}\label{Nac}
Let $X\subset\Omega$ be an open submanifold whose closure $\bar X$ is a smooth submanifold with boundary  and let $s\in\Gamma(\bar X,E)$ be a smooth section defined on $\bar X$.

Let $O\subset \Omega$ be an open submanifold containing $X$ and $\sigma\in\Gamma(O,E)$ be a smooth section  such that $\sigma|_X=s|_X$. 

There exists an open subset $O'\subset O$  such that $\bar X\cap O'=\bar X\cap O$ and the triple $(s,O',\sigma|_{O'})$ satisfies {\bf P}$_k$ for any $k\geq 0$.
\end{lm}
\begin{proof}
Let $N$ be the normal line bundle of $\partial\bar X$,  let
$$
\begin{tikzcd}
N\ar[d,"\nu"']\ar[rd,     "f"]& \\
\partial\bar X \ar[r, hook]&\Omega	
\end{tikzcd}
$$ 
be a tubular neighborhood of $\partial\bar X$ in $\Omega$ \cite[Section 4.5]{Hir} and $N_+\subset N$ be the outer side of the normal bundle. Replacing $\Omega$ by $\bar X\cup f(N_+)$ and $O$ by its intersection with $\bar X\cup f(N_+)$,  we may suppose that there exists a continuous retraction $r:\Omega\to \bar X$.\vspace{2mm}

Fix a Riemannian metric $g$ on $\Omega$ and define $\delta:\Omega \to\R$ by $\delta(\omega)\edf d_g(\omega,\bar X\setminus O)$, where $d_g$ is the distance associated with $g$. It is a well known general fact (valid in the general framework of metric spaces) that such a function is 1-Lipschitz, so continuous. Moreover, since $\bar X\setminus O$ is closed, $\delta^{-1}(0)=\bar X\setminus O$.\vspace{2mm}

For any $r>0$ the set $V_r\edf \delta^{-1}([0,r))$ is an open neighborhood of $\bar X\setminus O$ and $\bigcap_{r>0} V_r=\bar X\setminus O$.
Let $(\rho_n)_n$, $(r_n)_n$  be decreasing sequences of positive numbers converging to 0 such that $\forall n\in\N,\ \rho_n< r_n$. For any $n\in\N$ let $\chi_n$ be a smooth $[0,1]$-valued function  on $\Omega$  such that 
$$V_{\rho_n}\subset \chi_n^{-1}(1)\subset  \mathrm{supp}(\chi_n)\subset  V_{r_n}. 
$$

Let $d_k$ be any metric on the  total space of the bundle $T^{*\otimes k}_\Omega\otimes E$ which is compatible with its topology. A natural way to obtain such a metric is to choose  the connections $\nabla$, $\nabla_E$ compatible with $g$ and an Euclidian structure $h$ on $E$ respectively, and to note that these data endow the total space of $T^{*\otimes k}_\Omega\otimes E$ with a  Riemannian metric $\g_k$ which makes the projection $T^{*\otimes k}_\Omega\otimes E\to \Omega$ a Riemannian submersion and the inclusions $T^{*\otimes k}_u\otimes E_u\hookrightarrow T^{*\otimes k}_\Omega\otimes E$  isometric embeddings. 

Note first that the sum
\begin{equation}\label{NacinoSum}
\sum_{k\geq 0} \chi_k(u) \ d_k \big(D^k\sigma(u)\,,\, D^ks(r(u))\big)	
\end{equation}
is locally finite on $O$; indeed, for a point $v\in O$, we have $\delta(v)>0$ and the function $\chi_k$ vanishes on $\delta^{-1}(\frac{1}{2}\delta(v),+\infty)$ (which is an open neighborhood of $v$) for any $k\in\N$ for which $r_k<\frac{1}{2}\delta(v)$.  Since $r$ is continuous, it follows that (\ref{NacinoSum}) defines a continuous function on $O$, so the subset
\begin{equation}\label{NaciFormula}
O'\edf \big\{u\in O|\ \sum_{k\geq 0} \chi_k(u) \ d_k \big(D^k\sigma(u)\,,\, D^ks(r(u))\big)<\delta(u)\big\}	
\end{equation}
of $O$ is open. For a point $x\in\bar X\cap O$ we have $r(x)=x$ and $D^k\sigma(x)=D^ks(x)$ for any $k\geq 0$ by (\ref{hypnew-k}), so $x\in O'$. Therefore $\bar X\cap O=\bar X\cap O'$ as claimed. \vspace{2mm}

We now prove that the triple $(s,O',\sigma|_{O'})$ satisfies {\bf P}$_\kappa$ for any $\kappa\geq 0$. Fix $\kappa\geq 0$ and let $(u_n)_n$ be a sequence of $O'$ such that $\lim_{n\to\infty} u_n=x\in \bar X\setminus O=X\setminus O'$.  Since $\lim_{n\to\infty} u_n\in \bar X\setminus O$, it follows $\lim_{n\to \infty} \delta(u_n)=0$, so there exists $n_\kappa\in\N$ such that $\delta(u_n)<\rho_\kappa$, hence $\chi_\kappa(u_n)=1$, for any $n\geq n_\kappa$.   Since $u_n\in O'$, formula (\ref{NaciFormula}) shows that
$$
\forall n\geq n_\kappa,\ d_\kappa \big(D^\kappa\sigma(u_n)\,,\, D^\kappa s(r(u_n))\big)<\delta(u_n),
$$
so
\begin{equation}\label{limdkappa}
\lim_{n\to\infty} d_\kappa \big(D^\kappa\sigma(u_n)\,,\, D^\kappa s(r(u_n))\big)=0.	
\end{equation}
 But $\lim_{n\to\infty} u_n=x$ implies $\lim_{n\to\infty} r(u_n)=r(x)=x$, so 
 \begin{equation}\label{limit=Dkappas(x)}
 \lim_{n\to\infty} D^\kappa s(r(u_n))=D^\kappa s(x),	
 \end{equation}
 because $D^\kappa s$ is continuous on $\bar X$. By (\ref{limdkappa}), (\ref{limit=Dkappas(x)}) we have $\lim_{n\to\infty} D^\kappa\sigma(u_n)=D^\kappa s(x)$, as claimed.

\end{proof}

\end{document}